\documentclass{amsart}
\usepackage{graphicx}
\usepackage{amscd}
\usepackage{amsmath}
\usepackage{amsthm}
\usepackage{amsfonts}
\usepackage{amssymb}
\usepackage{mathrsfs}
\usepackage{enumerate}
\usepackage[pagebackref,hypertexnames=false, colorlinks, citecolor=blue, linkcolor=red, pdfstartview=FitB]{hyperref}
\usepackage{amsrefs}

\usepackage[latin1]{inputenc}

\newcommand{\D}{\operatorname{\mathbb{D}}}

\newcommand{\N}{\operatorname{\mathbb{N}}}

\newcommand{\B}{\operatorname{\mathcal{B}}}

\newcommand{\hil}{\operatorname{\mathcal{H}}}

\DeclareMathOperator{\lat}{Lat}
\DeclareMathOperator{\alg}{Alg}

\newcommand{\ol}{\overline }

\DeclareMathOperator{\ran}{ran}
\newtheorem{lemma}{Lemma}[section]
\newtheorem{theorem}[lemma]{Theorem}
\newtheorem{proposition}[lemma]{Proposition}
\newtheorem{corollary}[lemma]{Corollary}
\theoremstyle{definition}

\begin{document}
\author{Rapha\"el Clou\^atre}
\address{Department of Pure Mathematics, University of Waterloo, 200 University Avenue West,
Waterloo, Ontario, Canada N2L 3G1} 
\email{rclouatre@uwaterloo.ca}
\title[Quasiaffine orbits of invariant subspaces ]
{Quasiaffine orbits of invariant subspaces for uniform Jordan operators }
\begin{abstract}
We consider the problem of classification of invariant subspaces for the class of uniform Jordan operators. We show that given two invariant subspaces $M_1$ and $M_2$ of a uniform Jordan operator $T=S(\theta)\oplus S(\theta)\oplus \ldots$, the subspace $M_2$ belongs to the quasiaffine orbit of $M_1$ if and only if the restrictions $T|M_1$ and $T|M_2$ are quasisimilar and the compression $T_{M_2^\perp}$ can be injected in the compression $T_{M_1^\perp}$. Our result refines previous work on the subject by Bercovici and Smotzer.
\end{abstract}

\maketitle
\section{Introduction}
Let $T_1:\hil_1\to \hil_1$ and $T_2:\hil_2\to \hil_2$ be bounded linear operators on Hilbert
spaces. If $M_1$ and $M_2$ are invariant subspaces for $T_1$ and
$T_2$ respectively  (that is $M_1\subset \hil_1$ and $M_2\subset \hil_2$ are closed subspaces such that $T_1M_1\subset M_1$ and $T_2M_2\subset M_2$), we say that $M_1$ is a \textit{quasiaffine transform} of $M_2$ if there
exists a bounded injective operator with dense range $X:\hil_1\to
\hil_2$ such that $XT_1=T_2 X$ and $\ol{XM_1}=M_2$. We write $M_1\prec M_2$ when $M_1$ is
a quasiaffine transform of $M_2$. In that case, we also say that $M_2$ lies in the \textit{quasiaffine orbit} of $M_1$. When $M_1\prec M_2$ and $M_2\prec M_1$, we say that $M_1$ and $M_2$ are \textit{quasisimilar} and write $M_1\sim M_2$. Quasisimilarity is clearly an
equivalence relation on the class of pairs of the form $(T,M)$, where $M$ is an invariant subspace for the bounded linear operator $T$. In \cite{bercovici1991} (see Problem 5.2), Bercovici raised the basic question underlying our present
investigation: describe the quasiaffine orbit of a given
invariant subspace for an operator of class $C_0$ (the definition of which will be recalled in Section 2).

Related results for general operators of class $C_0$ can be found in \cite{bercovici1991},
where it is proved that the quasisimilarity class of an invariant
subspace is determined by the quasisimilarity class of the restriction $T|M$ if and only if $T$ has a certain finiteness property (namely property (Q) introduced in \cite{uchiyama1979}).
Nilpotent operators of finite multiplicity have been considered in \cite{li1999}. In that context, it was proved that the
quasisimilarity class of $M$ is determined by the quasisimilarity classes of the restriction
$T|M$ and of the
compression $T_{M^\perp}$ when either of those operators has
multiplicity one. In addition, the authors of \cite{li1999} considered a
combinatorial object (a sequence of partitions) known as a
\textit{Littlewood-Richardson sequence} which encodes the
relationships that must hold between the Jordan models of $T$, $T|M$
and $T_{M^\perp}$ (see also \cite{bercovici1998},
\cite{bercovici2005} and \cite{li1998}). Using these objects, they prove that for multiplicity
at least three, the quasisimilarity classes of $T|M$ and
$T_{M^\perp}$ are not enough to determine the quasisimilarity class
of $M$.  From a slightly different point of view, it was proved in \cite{bercovici1998} that if $M_1$ and $M_2$ are cyclic invariant subspaces for $T$, then they must be quasiaffine transforms of a common cyclic invariant subspace $N$  (in other words, $M_1$ and $M_2$ lie in the same weakly
quasiaffine orbit) whenever the restrictions (respectively, the compressions) of $T$ to $M_1$ and $M_2$ are quasisimilar.

The objects we will be concerned with in this work are the so-called uniform Jordan operators: that is operators of the form 
$$
T=S(\theta)\oplus
S(\theta)\oplus \ldots
$$
These operators are interesting since any $C_0$ contraction is the compression of a uniform Jordan operator, where $\theta$ is the minimal function of $T$. This well-known fact follows easily from considerations related to the the minimal isometric dilation of $T$, and we refer to the reader to \cite{nagy2010} for greater detail. In addition, uniform Jordan operators appear to be more amenable  and our understanding of the quasisimilarity classes of their invariant subspaces is significantly better, thanks to the pioneer work  of Bercovici and Tannenbaum (see \cite{bercovici1991tann}). Indeed, motivated by interpolation problems from \cite{bercovici1991NP} and \cite{bercovici1991lift}, they considered the case where the Jordan operator $T$ has finite multiplicity and established that $M_1 \sim M_2$ if and only if $T|M_1\sim T|M_2$.
Moreover, it was observed that for the operator $T=S(z^2)\oplus S(z)$ this
classification breaks down, so the corresponding result fails if $T$
is not uniform. Later on, it was proved in \cite{bercovici1991} that this
classification holds for a uniform Jordan operator $T$ if and only if $T|M$ satisfies
property (P), another finiteness property which is stronger than the aforementioned property (Q). In general, the quasisimilarity class of an invariant
subspace for a uniform Jordan operator is determined by the
quasisimilarity classes of the restriction $T|M$ and of the
compression $T_{M^\perp}$ (see \cite{bercovici1996}).

We focus in this paper on the weaker notion of quasiaffine orbit. After presenting the necessary preliminaries in Section 2, we prove in Section 3 our main theorem (Theorem \ref{mainjordan}) which gives a characterization of these orbits for uniform Jordan operators (thus refining the main result of \cite{bercovici1996}). Finally, in Section 4, we examine the question of weakening the condition on the operator $T$ to merely having a uniform Jordan model.


\section{Background and preliminaries}
We give here some background concerning operators of class $C_0$.
Let $H^\infty$ be the algebra of bounded holomorphic functions on the open unit disc $\D$. Let $\hil$ be a Hilbert space and $T$ a bounded linear operator on $\hil$, which we indicate by $T\in\B(\hil)$. The operator $T$ is said to be of \textit{class $C_0$} if there exists an algebra homomorphism $\Phi: H^\infty \to \B(\hil)$ with the following properties:
\begin{enumerate}[(i)]
    \item $\|\Phi(u)\|\leq \|u\|$ for every $u\in H^\infty$
    \item $\Phi(p)=p(T)$ for every polynomial $p$
    \item $\Phi$ is continuous when $H^\infty$ and $\B(\hil)$ are given their respective weak-star topologies
    \item $\Phi$ has non-trivial kernel.
\end{enumerate}
We use the notation $\Phi(u)=u(T)$, which is the Sz.-Nagy--Foias $H^\infty$ functional calculus.
It is known that $\ker \Phi=m_T H^\infty$ for some inner function $m_T$ called the \textit{minimal function} of $T$. The minimal function is uniquely determined up to a scalar factor of absolute value one. 

A set $E\subset \hil$ is said to be \textit{cyclic} for $T$ if $\hil=\bigvee_{n=0}^\infty T^n E$. The \textit{multiplicity} of the operator $T$ is the smallest
cardinality of a cyclic set. If $T$ has multiplicity one, it is said to be \textit{multiplicity-free}.

Let $H^2$ denote the Hilbert space of functions $f(z)=\sum_{n=0}^\infty a_n z^n$ holomorphic in $\D$ equipped with the norm
$$
\|f\|=\left(\sum_{n=0}^\infty |a_n|^2\right)^{1/2}.
$$ 
For any inner function $\theta\in H^\infty$, the space $H(\theta)=H^2\ominus \theta H^2$ is  invariant for $S^*$, the adjoint of the shift operator $S$ on $H^2$. The operator $S(\theta)$ defined by $S(\theta)^*=S^*|(H^2\ominus \theta H^2)$ is called a \textit{Jordan block}; it is of class $C_0$ with minimal function $\theta$. We state some useful properties of these operators. Given functions $u,v\in H^\infty$, we say that $u$ \textit{divides} $v$ and write $u|v$ if there exists a function $w\in H^\infty$ such that $v=wu$.

\begin{proposition}[\cite{bercovici1988} Proposition 3.1.10]\label{jordanblockprops}
Let $\theta\in H^\infty$ be an inner function.
\begin{enumerate}
    \item[\rm{(i)}] The operator $S(\theta)$ is multiplicity-free.
    \item[\rm{(ii)}] If $\phi\in H^\infty$ is an inner divisor of $\theta$, then $\phi H^2\ominus \theta H^2$ is an invariant subspace for 			$S(\theta)$. In fact,
    $$
    \phi H^2\ominus \theta H^2=\ran \phi(S(\theta))=\ker (\theta/\phi)(S(\theta)).
    $$
    Conversely, any invariant subspace for $S(\theta)$ is of this form.
    \end{enumerate}
\end{proposition}

A more general family of operators are the so-called \textit{Jordan operators}. We will define them here in the case where the Hilbert space on which they act is separable. These operators are of the form $\bigoplus_{n=0}^\infty S(\theta_n)$ where $\{\theta_n\}_{n=0}^\infty$ is a sequence of inner functions satisfying $\theta_{n+1}|\theta_n$ for $n\geq 0$. In case where $\theta_n=\theta$ for every $n\geq0 $ for some fixed inner function $\theta\in H^\infty$, then the operator $T=\bigoplus_{n=0}^\infty S(\theta)$ is called a  \textit{uniform Jordan operator}.

Recall that a bounded injective linear operator with dense range is called a \textit{quasiaffinity}. Two operators $T_1\in \B(\hil_1)$ and $T_2\in \B(\hil_2)$ are said to be \textit{quasisimilar} if there exist quasiaffinities $X:\hil_1\to \hil_2$ and $Y:\hil_2 \to \hil_1$ such that $XT_1=T_2 X$ and $T_1 Y=YT_2$. We use the notation $T_1\sim T_2$ to indicate that $T_1$ and $T_2$ are quasisimilar. The Jordan operators are of fundamental importance in the study of operators of class $C_0$ as the following theorem illustrates.

\begin{theorem}[\cite{bercovici1988} Theorem 3.5.1]\label{existencejordan}
For any operator $T$ of class $C_0$ acting on a separable Hilbert
space there exists a unique Jordan operator $J$ such that $T$ and
$J$ are quasisimilar.
\end{theorem}

The operator $J$ in the previous theorem is called the \textit{Jordan model} of $T$.
We now collect some facts about invariant subspaces for operators of class $C_0$.

\begin{proposition}[\cite{bercovici1988} Theorem 3.2.13, Theorem 3.3.8]\label{prop1.10}
Let $T\in \B(\hil)$ be an operator of class $C_0$. The following
statements are equivalent:
\begin{enumerate}
    \item[\rm{(i)}] $T$ is multiplicity-free
    \item[\rm{(ii)}] $\{T\}'$ is commutative
    \item[\rm{(iii)}] For every inner divisor $\theta$ of $m_T$ there exists a
    unique invariant subspace $K\subset \hil$ for $T$ such that $m_{T|K}=
    \theta$. In fact, $K=\ker \theta(T)=\ol{\ran (m_T/\theta)(T)}$.
    \end{enumerate}
\end{proposition}

Given a subset $E\subset \B(\hil)$, we denote its commutant by  
$$
E'=\{X\in \B(\hil):XT=TX \text{ for every }T\in E\}.
$$
We denote by $\lat(T)$ the collection of invariant subspaces for an operator $T$, and by $\alg \lat(T)$ the algebra of operators $X$ such that $XM\subset M$ for every $M\in \lat(T)$.

\begin{theorem}[\cite{bercovici1988} Theorem 4.1.2]\label{alglat}
For an operator $T$ of class $C_0$, we have $\alg \lat(T)\cap \{T\}'=\{T\}''$.
\end{theorem}

Let us recall a relation which is weaker than that of quasisimilarity. Given $T_1\in \B(\hil_1)$ and $T_2\in \B(\hil_2)$, we say that $T_1$ can be \textit{injected} in $T_2$ if there exists an injective operator $X:\hil_1\to \hil_2$ such that $XT_1=T_2X$. We indicate the fact that $T_1$ can be injected in $T_2$ by $T_1\prec^i T_2$. If in addition $X$ has dense range, we say that $T_1$ is a \textit{quasiaffine transform} of $T_2$ and we write $T_1\prec T_2$.

\begin{theorem}[\cite{bercovici1988} Proposition 3.5.31, Proposition 3.5.32]\label{injection}
Let $T_1$ and $T_2$ be two operators of class $C_0$. Then, the following are equivalent:
\begin{enumerate}
    \item[\rm{(i)}] $T_1 \prec T_2$
    \item[\rm{(ii)}] $T_1\prec^i T_2$ and $T_2 \prec^i T_1$
    \item[\rm{(iii)}] $T_1\sim T_2$.
\end{enumerate}
Moreover, $T_1\prec^i T_2$ if and only if $T_1^* \prec^i T_2^*$. If $\bigoplus_{n=0}^\infty S(\theta^{(1)}_n)$ and $\bigoplus_{n=0}^\infty S(\theta^{(2)}_n)$ are the Jordan models of $T_1$ and $T_2$ respectively, then $T_1\prec^i T_2$ if and only if $\theta^{(1)}_n$ divides $\theta^{(2)}_n$ for every $n\geq 0$.
\end{theorem}

Given an invariant subspace $M$ for an operator
$T$, we denote by $T_{M^\perp}$ the compression
$P_{M^\perp}T|M^\perp$. The following two results concerning uniform Jordan operators are from \cite{bercovici1996}. 

\begin{proposition}[\cite{bercovici1996} Proposition 2.1]\label{divrel}
Let $T=\bigoplus_{n=0}^\infty S(\theta)$ and $M$ be an invariant subspace for
$T$. Assume that $\bigoplus_{n=0}^\infty S(\phi_n)$ and
$\bigoplus_{n=0}^\infty S(\psi_n)$ are the Jordan models of $T|M$
and $T_{M^{\perp}}$ respectively. Then,
\begin{enumerate}
    \item[\rm{(i)}] $\phi_0$ and $\psi_0$ divide $\theta$
    \item[\rm{(ii)}] $\theta$ divides $\phi_m \psi_n$ for every $m,n\geq 0$.
\end{enumerate}
\end{proposition}

\begin{theorem}[\cite{bercovici1996} Theorem 2.5]\label{canspace}
Let $T=\bigoplus_{n=0}^\infty S(\theta)$ and $M$ be an invariant subspace for
$T$. Assume that $\bigoplus_{n=0}^\infty S(\phi_n)$ and
$\bigoplus_{n=0}^\infty S(\psi_n)$ are the Jordan models of $T|M$
and $T_{M^{\perp}}$ respectively. Then, $M$ is quasisimilar to
$$
\bigoplus_{n=0}^\infty(\gamma_{n} H^2\ominus \theta H^2)
$$
where $\gamma_{n}=\theta/\phi_{n/2}$ for $n$ even, and
$\gamma_{n}=\psi_{(n-1)/2}$ for $n$ odd.
\end{theorem}

Let us close this section by proving an elementary fact which
motivates our main result.

\begin{proposition}\label{easydirection}
Let $T_1\in \B(\hil_1), T_2\in \B(\hil_2)$ be operators of class $C_0$ and let
$M_1\subset \hil_1, M_2\subset \hil_2$ be invariant subspaces for $T_1$ and
$T_2$ respectively. Assume that $M_1\prec M_2$. Then, $T_1|M_1\sim T_2|M_2$ and
$(T_2)_{M_2^\perp}\prec^i (T_1)_{M_1^\perp}$.
\end{proposition}
\begin{proof}
By assumption, there exists a quasiaffinity $X:\hil_1\to \hil_2$
such that $XT_1=T_2X$ and $\ol{XM_1}=M_2$. It follows that $X|M_1$ implements a quasiaffine transform
between $T_1|M_1$ and $T_2|M_2$ and Theorem \ref{injection} implies that $T_1|M_1\sim T_2|M_2$. Now, we have $X^*M_2^\perp\subset M_1^\perp$ so if we set $A=P_{M_2^\perp}X|M_1^\perp$ then
$$
\ker A^*=\{h\in M_2^\perp: X^*h\in M_1\}=0.
$$
In addition,
$$
P_{M_2^\perp}XP_{M_1^\perp}=P_{M_2^\perp}X
$$
and
$$
P_{M_2^\perp}T_2P_{M_2^\perp}=P_{M_2^\perp}T_2.
$$
We infer that
\begin{align*}
A(P_{M_1^\perp}T_1|{M_1^\perp})&=P_{M_2^\perp}XT_1|M_1^\perp\\
&=P_{M_2^\perp}T_2X|M_1^\perp\\
&=(P_{M_2^\perp}T_2|{M_2^\perp})A.
\end{align*}
and thus
$$
(P_{M_2^\perp}T_2|{M_2^\perp})^*\prec^i(P_{M_1^\perp}T_1|{M_1^\perp})^*.
$$
By Theorem \ref{injection}, we get
$$
P_{M_2^\perp}T_2|{M_2^\perp}\prec^iP_{M_1^\perp}T_1|{M_1^\perp}
$$
and the proof is complete.
\end{proof}

Our main result shows that the converse of the previous proposition holds when $T_1=T_2$ are uniform Jordan operators.

\section{Uniform Jordan operators}
Let us start with an elementary lemma.

\begin{lemma}\label{open}
Let $\phi,\psi\in H^\infty$ be inner divisors of the inner function
$\theta\in H^\infty$. Assume that $\theta/\phi$ divides $\psi$ and
set $\omega=\psi/(\theta/\phi).$ Then for every $g\in \psi H^2\ominus \theta H^2$ we can find $f\in
(\theta/\phi)H^2\ominus \theta H^2$ such that $\omega(S(\theta))f=g$
and
$\|f\|=\|g\|.$
\end{lemma}
\begin{proof}
Fix $g\in \psi H^2\ominus \theta H^2$. By Proposition \ref{jordanblockprops}, we have
$$
\omega(S(\theta))\left( (\theta/\phi)H^2\ominus \theta
H^2\right)= \left(\omega\theta/\phi\right)(S(\theta))\left(
H^2\ominus \theta H^2\right)=\psi H^2\ominus \theta H^2.
$$
We can thus find $f_0\in(\theta/\phi)H^2\ominus \theta H^2$ such that
$\omega(S(\theta))f_0=g$. We set $f=P_{H(\theta/\omega)}f_0$.
Then, there exists a function $h\in H^2$ such that $f=f_0+(\theta/\omega)h$,
whence 
$
f=f_0+(\theta/\phi)(\theta/\psi)h\in (\theta/\phi)H^2
$
and thus $f\in (\theta/\phi)H^2\ominus \theta H^2$ since $f\in H(\theta/\omega)\subset H(\theta)$. Moreover,  $\omega f\in H(\theta)$ and
$$
g=\omega(S(\theta))f=P_{H(\theta)}\omega f=\omega f
$$
and since $\omega$ is an inner function, we have that $\|g\|=\|f\|$.
\end{proof}

The following two lemmas provide the crucial tool for the proof of our main result.

\begin{lemma}\label{qa}
Let $\theta \in H^\infty$ be an inner function. Let $\hil=\bigoplus_{n=0}^\infty H(\theta)$ and $T=\bigoplus_{n=0}^\infty S(\theta)$. Let
$(\omega_n)_{n=0}^\infty\in H^\infty$ be a sequence of inner divisors of
$\theta$ and let $\{c_n\}_{n=0}^\infty$ be a bounded sequence of positive numbers.
Define
$$
X:H(\theta)\oplus \hil\to H(\theta)\oplus \hil
$$
as follows
$$
X(g\oplus (f_n)_n )= \left(g+\sum_{n=0}^\infty \frac{1}{n+1}\omega_n(S(\theta))f_n \right)\oplus (c_nf_n)_n.
$$
Then,  $X$ is a quasiaffinity which commutes with $S(\theta)\oplus T$.
\end{lemma}

\begin{proof}
It is immediate that $X$ is injective, and a routine calculation shows that $X$ is bounded.
Pick now $G\oplus
(F_n)_n\in H(\theta)\oplus \hil$. For $m\geq 0$, define
$$
g_m=G-\sum_{n=0}^m \frac{1}{(n+1)c_n}\omega_n(S(\theta))F_n
\in H(\theta)
$$
and
$$
u_m=\left( \frac{1}{c_0}F_0,
\frac{1}{c_1}F_1,\ldots,\frac{1}{c_m}F_m,0,\ldots\right)\in \hil.
$$
We get that
$$
X(g_{m}\oplus u_m)=G\oplus ( F_1,\ldots,F_m, 0,\ldots)
$$
and thus 
$$
\lim_{m\to \infty}X(g_{m}\oplus u_m)=G\oplus
(F_n)_n.
$$
This shows that $X$ has dense range. Finally,  we have
\begin{align*}
(S(\theta)\oplus T)X(g\oplus
(f_n)_n)&=\left(S(\theta)g+\sum_{n=0}^\infty
\frac{1}{n+1}S(\theta)\omega_n(S(\theta))f_n\right)\oplus (c_n
S(\theta)f_n))_n\\
&=\left(S(\theta)g+\sum_{n=0}^\infty
\frac{1}{n+1}\omega_n(S(\theta))S(\theta)f_n\right)\oplus (c_n
S(\theta)f_n)_n\\
&=X(S(\theta)\oplus T)(g\oplus (f_n)_n)
\end{align*}
which completes the proof.
\end{proof}

\begin{lemma}\label{invsub}
Let $\psi_1,\psi_2\in H^\infty$ be inner functions and
$(\phi_n)_{n=0}^\infty \in H^\infty$ be a sequence of inner functions. Assume the following divisibility relations:
\begin{enumerate}
\item[\rm{(i)}] $\psi_2$ divides $\psi_1$
\item[\rm{(ii)}] $\phi_n$ divides $\theta$ for every $n\geq 0$  and $\psi_1$ divides $\theta$
\item[\rm{(iii)}] $\phi_{n+1}$ divides $\phi_n$ for every $n\geq 0$
\item[\rm{(iv)}] $\theta/\phi_n$ divides $\psi_2$ for every $n\geq 0$.
\end{enumerate}
Let $\hil=\bigoplus_{n=0}^\infty H(\theta)$ and let
$\omega_n=\psi_2/(\theta/\phi_n)$ for each $n\geq 0$.  Define
$$
X:H(\theta)\oplus \hil\to H(\theta)\oplus \hil
$$
as
$$
X(g\oplus (f_n)_n )= \left(g+\sum_{n=0}^\infty \frac{1}{n+1}\omega_n(S(\theta))f_n \right)\oplus (c_nf_n)_n,
$$
where $\{c_n\}_{n=0}^\infty$ is a sequence of positive numbers satisfying 
\begin{equation}\label{convergence}
\lim_{m\to \infty}(m+1) c_m \left(
\sum_{n=0}^{m-1}\frac{1}{|(n+1)c_n|^2}\right)^{1/2}=\lim_{m\to
\infty}(m+1)c_m =0.
\end{equation}
Then, we have $\overline{X(N_{\psi_1}\oplus M)}=N_{\psi_2}\oplus M$, where
$$
M=\bigoplus_{n=0}^\infty \left((\theta/\phi_n)H^2\ominus \theta
H^2\right)\subset  \hil,
$$
$$
N_{\psi_1}=\psi_1 H^2\ominus \theta H^2\subset H(\theta),
$$
$$
N_{\psi_2}=\psi_2 H^2\ominus \theta H^2 \subset H(\theta).
$$
\end{lemma}

\begin{proof}
First note that $N_{\psi_1}\subset N_{\psi_2}$ since $\psi_2$ divides $\psi_1$.  Moreover, it follows from Proposition \ref{jordanblockprops} that for every $n\geq 0$
\begin{equation}\label{rangewn}
\omega_n(S(\theta))\left((\theta/\phi_n)H^2\ominus \theta H^2\right)=\psi_2 H^2 \ominus \theta H^2.
\end{equation}
Therefore, we have $X(N_{\psi_1}\oplus M)\subset N_{\psi_2}\oplus
M.$

Let now $G\oplus (F_n)_n\in N_{\psi_2}\oplus M$, in other words
$G\in \psi_2 H^2\ominus \theta H^2$ and $F_n\in (\theta/\phi_n)H^2\ominus \theta H^2$
for every $n\geq 0$. It follows from (\ref{rangewn}) that 
$$
G-\sum_{n=0}^{m}\frac{1}{(n+1)
c_n}\omega_n(S(\theta))F_n\in \psi_2 H^2\ominus \theta H^2
$$
for every $m\geq 0$. Consequently, for every $m\geq 1$  using Lemma \ref{open} we can find a function 
$h_m\in (\theta/\phi_m)H^2\ominus \theta H^2$ such that
\begin{equation}\label{G}
\frac{1}{m+1}\omega_{m}(S(\theta))h_{m}=G-\sum_{n=0}^{m-1}\frac{1}{(n+1)
c_n}\omega_n(S(\theta))F_n
\end{equation}
and
\begin{equation}\label{Gnorm}
\|h_m\|=(m+1)\left\| G-\sum_{n=0}^{m-1}\frac{1}{(n+1)
c_n}\omega_n(S(\theta))F_n\right\|.
\end{equation}

Using equation (\ref{G}) and the definition of $X$ yields
$$
X\left(0\oplus
\left(\frac{1}{c_0}F_0,\ldots,\frac{1}{c_{m-1}}F_{m-1},h_m,0,\ldots\right)\right)=G\oplus(F_0,\ldots,F_{m-1},
c_m h_m,0\ldots).
$$
Therefore,
\begin{align}\label{zero}
&\left\| G\oplus (F_n)_n-X\left(0\oplus
\left(\frac{1}{c_0}F_0,\ldots,\frac{1}{c_{m-1}}F_{m-1},h_m,0,\ldots\right)\right)\right\|^2 \notag\\
&=\|F_m-c_mh_m\|^2+\sum_{n=m+1}^{\infty}\|F_n\|^2.
\end{align}
We now proceed to show that this last quantity tends to zero as $m$ tends to infinity. Note that $\omega_n(S(\theta))$ is a contraction for every $n\geq 0$ since $\omega_n\in H^\infty$ is an inner function. Using equation (\ref{Gnorm}) and a standard application of the Cauchy-Schwarz inequality, we get for every $m\geq 1$ that
$$
\|h_m\|\leq (m+1) \left(
\|G\|+\|(F_n)_n\|\left(\sum_{n=0}^{m-1}\frac{1}{|(n+1) c_n|^2}
\right)^{1/2}\right).
$$
Hence, we have that $c_m\|h_m\|\to 0$ as $m\to \infty$. Indeed
$$
c_m\|h_m\|\leq (m+1)c_m  \|G\|+(m+1)c_m
\|(F_n)_n\|\left(\sum_{n=0}^{m-1}\frac{1}{|(n+1)c_n|^2} \right)^{1/2}
$$
and the right-hand side goes to zero as $m$ goes to infinity in view of (\ref{convergence}), which proves that (\ref{zero}) tends to zero as $m$ tends to infinity. We conclude that
$$
\lim_{m\to \infty }X\left(0\oplus
\left(\frac{1}{c_0}F_0,\ldots,\frac{1}{c_{m-1}}F_{m-1},h_m,0,\ldots\right)\right)=G\oplus
(F_n)_n
$$
so that $\ol{X(N_{\psi_1}\oplus M)}= N_{\psi_2}\oplus
M.$
\end{proof}
The reader will notice that a sequence $\{c_n\}_n$ satisfying equation (\ref{convergence}) is easily constructed by induction, for instance. We can now establish our main result.

\begin{theorem}\label{mainjordan}
Let $T=\bigoplus_{n=0}^\infty S(\theta)$ and $M_1,M_2$ be invariant subspaces for
$T$. Then $M_1\prec M_2$ if and only if $T|M_1\sim T|M_2$ and $T_{M_2^\perp}\prec^i T_{M_1^\perp}$. 
\end{theorem}
\begin{proof}
One direction follows from Proposition \ref{easydirection}. Assume therefore that $T|M_1\sim T|M_2$ and
$T_{M_2^\perp}\prec^i T_{M_1^\perp} $. Let $\bigoplus_{n=0}^\infty S(\phi_j)$ be the common Jordan model of $T|M_1$ and $T|M_2$, and let
$\bigoplus_{n=0}^\infty S(\psi_j)$ and $\bigoplus_{n=0}^\infty S(\tau_j)$ be the Jordan models of $T_{M_1^\perp}$ and $T_{M_2^\perp}$ respectively. 
Define
$$
M_1'=\bigoplus_{n=0}^\infty\gamma_{n} H^2\ominus \theta H^2,
$$
$$
M_2'=\bigoplus_{n=0}^\infty\delta_{n} H^2\ominus \theta H^2
$$
where $\gamma_{n}=\theta/\phi_{n/2}$ and $\delta_{n}=\theta/\phi_{n/2}$ for $n$ even, while $\gamma_{n}=\psi_{(n-1)/2}$ and $\delta_{n}=\tau_{(n-1)/2}$ for $n$ odd. By Theorem \ref{canspace}, we have that $M_1\sim M_1'$ and $M_2\sim M_2'$, so it suffices to show that $M_1'\prec M_2'$.

Let $\hil=\bigoplus_{n=0}^\infty H(\theta)$ and let $F: \N\times \N\to \N$ be the classical bijective pairing function defined as
$$
(n,m)\mapsto \frac{1}{2}(n+m)(n+m+1)+n.
$$
where $\N=\{0,1,2,\ldots\}$. Define
$$
V: \hil \to \bigoplus_{n=0}^\infty \left(H(\theta)\oplus \hil\right)
$$
as follows
$$
V(f_0,g_0,f_1,g_1,\ldots)=\bigoplus_{n=0}^\infty \left(g_n \oplus
(f_{F(n,m)})_{m=0}^\infty \right).
$$
It is clear that $V$ is an isometry with isometric inverse given by
$$
\bigoplus_{n=0}^\infty \left(g_n \oplus
(f_{n,m})_{m=0}^\infty \right)\mapsto (f_{F^{-1}(0)},g_0,f_{F^{-1}(1)},g_1,\ldots).
$$
Hence, $V$ is unitary and a straightforward verification shows that $V$
satisfies
$$
VT=\left(\bigoplus_{n=0}^\infty (S(\theta)\oplus T)\right)V.
$$
In addition, we have
$$
VM_1' =\bigoplus_{n=0}^\infty \left( (\psi_n H^2 \ominus \theta H^2)\oplus\bigoplus_{m=0}^\infty \left(\frac{\theta}{\phi_{F(n,m)}}H^2\ominus \theta H^2\right)
\right)
$$
and likewise
$$
VM_2'= \bigoplus_{n=0}^\infty \left( (\tau_n H^2 \ominus \theta H^2)\oplus\bigoplus_{m=0}^\infty \left(\frac{\theta}{\phi_{F(n,m)}}H^2\ominus \theta H^2\right)
\right).$$
Now, it easily verified that
$
F(n,m+1)\geq F(n,m)
$
for every $n,m\geq 0$. In particular, we have that
$\phi_{F(n,m+1)}$ divides $\phi_{F(n,m)}$
for every $n,m\geq 0$. 
In addition, we see by Proposition \ref{divrel} that $\theta/\phi_{F(n,m)}$ divides $\tau_p$ for every $n,m,p\geq 0$. Finally, since $T_{M_2^\perp}\prec^i T_{M_1^\perp}$ Theorem \ref{injection} implies that $\tau_n$ divides $\psi_n$ for every $n\geq 0$. We are therefore in position to apply Lemma \ref{qa} and Lemma \ref{invsub} to get for each $n\geq 0$ a quasiaffinity $X_n$ commuting with $S(\theta)\oplus T$ and satisfying 
\begin{align*}
&\overline{X_n\left( (\psi_{n}H^2 \ominus \theta H^2)\oplus\bigoplus_{m= 0}^\infty \left(\frac{\theta}{\phi_{F(n,m)}}H^2\ominus \theta H^2\right)
\right)}\\
&=(\tau_n H^2 \ominus \theta H^2)\oplus \bigoplus_{m=0}^\infty \left(\frac{\theta}{\phi_{F(n,m)}}H^2\ominus \theta H^2\right).
\end{align*}
If we put
$$
Y=V^*\left(\bigoplus_{n=0}^\infty \frac{X_n}{\|X_n\|}\right)V,
$$
then it is then easy to
check that $\overline{YM_1'}=M_2'$ and that $Y$ is a quasiaffinity
commuting with $T$. Hence, $M_1'\prec M_2'$ and we are done.
\end{proof}

\section{Uniform Jordan models}
The aim of this section is to relax the assumption on $T$ being a uniform Jordan operator. Namely, we aim at getting a result analogous to Theorem 
\ref{mainjordan} in the case where $T$ is merely quasisimilar to
a uniform Jordan operator. We first need a preliminary fact.

\begin{lemma}\label{alglatqa}
Let $T$ be an operator of class $C_0$ and let $X\in \alg \lat (T)\cap \{T\}'$ be an injective operator. Then, $\ol{XM}=M$ for every $M\in \lat(T)$.
\end{lemma}
\begin{proof}
Let $M\in \lat(T)$. We decompose $M$ into cyclic subspaces: $M=\bigvee_{j=0}^\infty K_j$ where $K_j\in \lat(T)$ and
$T|K_j$ is multiplicity-free. Since $X\in \alg \lat (T)$, we
have $\ol{XK_j}\subset K_j$ for every $j \geq 0$. On the other hand, the fact that $X$ is
an injective operator commuting with $T$ implies that $T|K_j\sim
T|\ol{XK_j}$ for every $j\geq 0$. By Proposition \ref{prop1.10}, we
conclude that $\ol{XK_j}=K_j$ for every $j\geq 0$, which in turn implies
$\ol{XM}=\bigvee_{j=0}^\infty \ol{XK_j}=\bigvee_{j=0}^\infty K_j=M$.
\end{proof}

We now achieve the desired result under an
extra assumption.

\begin{theorem}\label{YXalg}
Let $T\in \B(\hil)$ be an operator of class $C_0$ with uniform Jordan
model  $J=\bigoplus_{n=0}^\infty S(\theta)$. Assume that we can find quasiaffinities
$$
X: \hil\to \bigoplus_{n=0}^\infty H(\theta), Y:\bigoplus_{n=0}^\infty H(\theta)\to \hil
$$
with the property that $XT=JX,$ $YJ=TY$ and $YX\in
\alg \lat(T)$. Let $M_1$ and $M_2$ be two invariant subspaces for $T$. Then,
$M_1\prec M_2$ if and only if $T|M_1\sim T|M_2$ and
$T_{M_2^\perp}\prec^i T_{M_1^\perp} $.
\end{theorem}

\begin{proof}
One direction follows from Proposition \ref{easydirection}. Assume therefore that $T|M_1\sim T|M_2$ and
$T_{M_2^\perp}\prec^i T_{M_1^\perp} $. Let $E_1=\ol{XM_1}$ and $E_2=\ol{XM_2}$. Notice that $YX$ is a
quasiaffinity commuting with $T$, so that by Lemma \ref{alglatqa} we
have $\ol{YXM_k}=M_k$ for $k=1,2$. In particular, this shows that
$\ol{YE_k}=M_k$ for $k=1,2$. It follows that $J|E_1\sim
T|M_1\sim T|M_2\sim J|E_2$. Moreover, we can write
$X^*E_k^\perp\subset M_k^\perp $ and $Y^* M_k^\perp \subset
E_k^\perp$ for $k=1,2$. Since $T_{M_2^\perp}$ can be injected into $T_{M_1^\perp}$, it
follows from Theorem \ref{injection} that $(T_{M_2^\perp})^*=T^*|M_2^\perp$ can be injected into $(T_{M_1^\perp})^*=T^*|M_1^\perp$,
so we can find an injective operator $Z:M_2^\perp \to M_1^\perp$
such that $Z(T^*|M_2^\perp)=(T^*|M_1^\perp)Z$. Set
$W=Y^*ZX^*|E_2^\perp:E_2^\perp \to E_1^\perp$, which is obviously
injective. Note that for $k=1,2$ we have
$$
(X^*|E_k^\perp) (J^*|E_k^\perp)=X^* J^*|E_k^\perp=T^*
X^*|E_k^\perp=(T^*|M_k^\perp)(X^*|E_k^\perp),
$$
$$
(Y^*|M_k^\perp) (T^*|M_k^\perp)=Y^* T^*|M_k^\perp=J^*
Y^*|M_k^\perp=(J^*|E_k^\perp)(Y^*|M_k^\perp),
$$
whence
\begin{align*}
W(J^*|E_2^\perp)&=Y^*ZX^*J^*|E_2^\perp\\
&=Y^* Z T^* X^* |E_2^\perp\\
&=Y^* Z (T^*|M_2^\perp) X^* |E_2^\perp\\
&=Y^* (T^*|M_1^\perp)Z X^*|E_2^\perp\\
&=(J^* Y^*|M_1^\perp)Z  X^*|E_2^\perp\\
&=(J^*|E_1^\perp)Y^*ZX^*|E_2^\perp\\
&=(J^*|E_1^\perp)W.
\end{align*}
Thus $J^*|E_2^\perp$ can be injected into $J^*|E_1^\perp$,
whence $J_{E_2^\perp}$ can be injected into $J_{E_1^\perp}$ via another application of Theorem \ref{injection}. By
Theorem \ref{mainjordan}, we can find a quasiaffinity $A\in \{J\}'$ such that $\ol{AE_1}=E_2$. Define finally $B=YAX$ which is
clearly another quasiaffinity. We then have
$$BT=YAXT=YAJX=YJAX=TYAX=TB$$
and
$$
\ol{BM_1}=\ol{YAXM_1}=\ol{YAE_1}=\ol{YE_2}=M_2
$$
and the proof is complete.
\end{proof}
In closing, let us mention an instance where the extra assumption $YX\in \alg \lat (T)$ appearing in Theorem \ref{YXalg} is satisfied.

\begin{corollary}\label{cordiag}
Let $T_0\in \B(\hil)$ be a multiplicity-free operator of class $C_0$ and let
$T=\bigoplus_{n=0}^\infty T_0$. Given $M_1$ and $M_2$ two invariant
subspaces for $T$, we have that $M_1\prec M_2$ if and only if
$T|M_1 \sim T|M_2$ and $T_{M_2^\perp} \prec ^i T_{M_1^\perp} $.
\end{corollary}
\begin{proof}
Denote by $\theta$ the minimal function of $T_0$. By assumption, we
can find  quasiaffinities $X: \hil\to H(\theta)$ and $Y:H(\theta)\to \hil$
with the property that $XT_0=S(\theta)X,$ $YS(\theta)=T_0Y$. Define
$A=\bigoplus_{n=0}^\infty X$ and $B=\bigoplus_{n=0}^\infty Y$, which
are quasiaffinities intertwining $T$ with its Jordan model
$\bigoplus_{n=0}^\infty S(\theta)$. By Theorem \ref{YXalg}, we need
only show that $BA\in \alg \lat (T)$. 

Since $T_0$ is multiplicity-free, Proposition \ref{prop1.10} implies
that $\{T_0\}'$ is commutative, and thus $\{T_0\}''=\{T_0\}'$. Therefore,
$YX\in \{T_0\}''$ and $BA=\bigoplus_{n=0}^\infty YX$ then clearly
belongs to $\{T\}''$ since 
$$
\{T\}'=\left\{(C_{nm})_{n,m=0}^\infty: C_{nm}\in \{T_0\}'\right\}.
$$
By Theorem \ref{alglat}, we find that
$BA\in \alg \lat (T)$ which completes the proof.
\end{proof}

\section{Acknowledgements}
The author was supported by a NSERC PGS grant and would like to thank the referee for a careful and thorough reading of the original version of this paper.

\def\cprime{$'$} \def\cprime{$'$} \def\cprime{$'$} \def\cprime{$'$}

\end{document}